\documentclass[aap,MSNbibl,seceqn,nameyear,dvips]{arximspdf}
\usepackage{graphicx}
\doi{10.1214/13-AAP967} 
\volume{24}
\issue{5}
\pubyear{2014}
\firstpage{1918}
\lastpage{1945}

\makeatletter
\newcommand{\rrVert}{\Vert}
\newcommand{\llVert}{\Vert}
\newtheorem{lemma}{Lemma}
\newtheorem{theorem}{Theorem}
\newtheorem{proposition}{Proposition}
\newtheorem{corollary}{Corollary}
\newproclaim{definition}{Definition}
\makeatother

\begin{document}
\begin{frontmatter}

\title{Sensitivity analysis for diffusion processes constrained to an orthant\thanksref{T1}}
\runtitle{Sensitivity analysis constrained diffusion processes}

\begin{aug}
\author[A]{\fnms{A. B.} \snm{Dieker}\corref{}\ead[label=e1]{ton.dieker@isye.gatech.edu}}
\and
\author[B]{\fnms{X.} \snm{Gao}\ead[label=e2]{xfgao@se.cuhk.edu.hk}}
\runauthor{A. B. Dieker and X. Gao}
\affiliation{Georgia Institute of Technology and Chinese University of Hong Kong}
\address[A]{H. Milton Stewart School of Industrial\\
\quad and Systems Engineering\\
Georgia Institute of Technology\\
Atlanta, Georgia 30332\\
USA\\
\printead{e1}} 
\address[B]{Department of Systems Engineering\\
\quad and Engineering Management\\
Chinese University of Hong Kong\\
Shatin, N.T.\\
Hong Kong\\
P.R. China\\
\printead{e2}}
\end{aug}
\thankstext{T1}{Supported by NSF Grants 0926308, 0727400 and 1030589.}

\received{\smonth{7} \syear{2011}}
\revised{\smonth{6} \syear{2013}}

%
\begin{abstract}
This paper studies diffusion processes constrained to the positive orthant
under infinitesimal changes in the drift.
Our first main result states that any constrained function and its
(left) drift-derivative
is the unique solution to an augmented Skorohod problem.
Our second main result uses this characterization to establish a basic
adjoint relationship for the stationary
distribution of the constrained diffusion process jointly with its
left-derivative process.
\end{abstract}

%

\end{frontmatter}

\section{Introduction}\label{secintroduction}
This paper is motivated by a desire to better understand the relation
between performance metrics
and control variables in a network with shared but limited resources.
We are specifically interested in service networks, where customers
seeking a certain
service may suffer from delays as a result of temporary insufficient
service capacity.
The control variables are the service capacities at the individual stations.
Many service processes can be modeled by stochastic (or queueing) networks,
and an important question is how resources should be allocated, given
random fluctuations in the arrivals and their interplay with
potentially random service times.
When planning horizons are long so that static allocation rules are required,
questions of this type are readily answered if the network has a
product-form structure
\citet{kleinrockbook}, \citet{MR1007270}. However, few results have been
obtained when this assumption
fails \citet{diekerghoshsquillantepub}, \citet{pollett2009}.
It is the goal of this paper to introduce new
tools in this context, which could be used in the context of both
sensitivity analysis and
system optimization.

We study diffusion processes and their ``derivatives,'' defined as
the change in the process under an infinitesimal change in the drift.
Although some of our results are stated more generally,
this paper focuses on diffusion processes for two reasons.
First, this framework allows us to explain key concepts in a tractable
yet relatively general setting.
Second, diffusion processes are rooted in heavy-traffic approximations for
stochastic networks, and the heavy-traffic assumption seems reasonable
in the context of resource allocation problems with systems operating
close to their capacity.
This paper studies the stationary distribution of diffusions and their
derivatives, as a proxy for
the long-term (steady-state) behavior. Although it is certainly
desirable to obtain time-dependent tools as well,
given the vast body of work on stationary results, making this
assumption is a natural first step.
The techniques developed in this paper are likely to be also relevant
in the time-dependent case.

We have two main results. The first is a statement on the behavior of
deterministic functions
under the well-known Skorohod reflection map with oblique reflection
(regulation), and states that
the map and its ``derivative'' are the unique solution to an augmented
version of the Skorohod problem.
Our proof of this result relies on recent insights into directional
derivatives by \citet{Mandel-ramanan},
which have been developed in the context of time-inhomogeneous systems
but are shown here to be useful
for sensitivity analysis as well.

Our second main result specializes to diffusion processes and studies
the stationary
distribution of solutions to the augmented Skorohod problem.
Given a \mbox{constrained} diffusion process $Z$ representing
the dynamics of the underlying stochastic network (i.e., the queue
lengths at each of the
stations), let the stochastic process $A$ represent the change in $Z$
under an infinitesimal change in the drift.
The two results combined say that the stationary distribution of the joint
processes $(Z, A)$ satisfies a kind of basic adjoint relation, which is
the analog
of the equation $\pi' Q =0$ for continuous-time Markov processes on a
discrete state space.
The proof relies on a delicate analysis of the jumps of $A$;
the process $A$ has jumps even if $Z$ is continuous.

The intuition behind the program carried out in this paper can be
summarized as follows.
Suppose $Z^\epsilon$ is a constrained diffusion process with drift
coefficient $\mu(\cdot)-\epsilon v$
in the interior of the orthant, where $v$ is an arbitrary nonnegative vector.
Suppose the processes $\{Z^\epsilon\}$ are driven by the same Brownian
motion for every $\epsilon\ge0$,
so that they are coupled.
The processes $Z\equiv Z^0$ and $Z^\epsilon$ are Markovian, and one
can therefore expect to be able to give
a basic adjoint relationship for their stationary distributions (should
they exist).
Moreover, $(Z,Z^\epsilon)$ and therefore $(Z,(Z-Z^\epsilon)/\epsilon)$
can be expected to be Markovian as a result of the coupling. Provided
one can make sense of
the pointwise limit $(Z,A)$ of $(Z,(Z-Z^\epsilon)/\epsilon)$ as
$\epsilon\to0+$,
one can expect that the distribution of $(Z,A)$ satisfies a similar
relationship. This results in an ``augmented''
basic adjoint relationship, which we state in Theorem~\ref{thmbarprimenew}.
The constrained diffusion processes studied in this paper are
pathwise solutions to stochastic differential equations with
reflection; see \citet{MR1110990}, \citet{MR2261058}.
We only consider left derivatives in this paper, although one could
develop similar tools and obtain similar results for right derivatives.
This would affect our two main results as follows.
On a sample-path level, the right derivative is
the left-continuous modification of the (right-continuous) left
derivative, see Section~\ref{secaugsk}
for a detailed discussion.
On a probabilistic level, studying the (left-continuous) right
derivative requires a different set of technical tools
since one ordinarily works with right-continuous stochastic processes.
We should expect that this change does not affect the stationary
distribution or
the basic adjoint relationship.

When carrying out the aforementioned approach,
we were surprised to find that,
even though $Z$ is known not to spend any time on low-dimensional faces,
it is critical to incorporate the jumps of $A$ when
$Z$ reaches those faces in order to formulate the basic adjoint relationship.

This work has the potential to lead to new numerical methods in the
context of optimization and sensitivity analysis for queueing networks,
which relieve or remove the need for computationally intensive or
numerically unstable
operations such as gradient estimation.
To explain, due to the division by $\epsilon$, any performance metric
of $(Z-Z^\epsilon)/\epsilon$
suffers from numerical instability issues for small $\epsilon>0$.
Researchers in stochastic optimization have developed several techniques
to mitigate this effect; see, for example, \citet{MR2331321}.
The approach taken in this paper is to analytically describe and
investigate the dynamics of the limit.
Our experience with state-of-the-art stochastic optimization
implementations in the context of
resource capacity management,
as documented in part in \citet{diekerghoshsquillantepub}, is that it
is computationally very costly
to obtain reliable gradient estimates and that the use of ``quick and
dirty'' estimates can have disastrous effects
on the compute time of a stochastic optimization procedure due to bias
and inherent random fluctuations.
Therefore, reliable (numerical) tools that give merely a rough idea of
the gradient can be desirable
and useful.
In particular, from an implementation perspective, heavy-traffic
gradient information can be
valuable even if a stochastic network is in moderate traffic.
(A light-traffic setting is not of prime interest
since one is typically interested in fine-tuning networks operating in
a regime where servers are idling relatively rarely.)

The framework of this paper is related to a body of literature known
as infinitesimal perturbation analysis \citeauthor{glassermanbook1991} (\citeyear{glassermanbook1991,glassermansurvey,MR1206536}), \citet{MR2273907}.
Infinite perturbation analysis also aims to perform sensitivity
analysis or
gradient estimation for performance metrics in (say) a queueing
network, and it does so
by formulating conditions under which an expectation and a
derivative operator can be interchanged.
Here, however, it is not our objective to seek such an interchange
involving a performance metric,
but instead we study the (whole)
stationary distribution of a stochastic process with its derivative process.


This paper is outlined as follows.
Section~\ref{seconedim} summarizes our approach in the
one-dimensional case, which serves
as a guide for our multidimensional results.
Section~\ref{secpreliminaries} discusses two technical preliminaries:
oblique reflection
maps and their derivatives.
In Section~\ref{secmainresults} we formulate our two main results.
Section~\ref{secproofthm2} is devoted to the proof of the first main result,
while Section~\ref{secproofthm3} gives the proof of the second main result.
A key role is played by jump measures, for which we obtain a
description in Section~\ref{appjumpmeasure}.
The appendices contain several technical digressions.

\subsection*{Notation}
For $J \in\mathbb{N}$, $\mathbb{R}^{J}$ denotes the $J$-dimensional
Euclidean space.
We denote\vadjust{\goodbreak} the space of real $n\times m$ matrices by $\mathbb M^{n\times m}$,
and the subset of nonnegative matrices by $\mathbb M^{n\times m}_+$.
All vectors are to be interpreted as column vectors, and we write
$M^j$ and $M_i$ for the $j$th column and the $i$th row of a matrix $M$,
respectively.
In particular, $v_i$ is the $i$th element of a vector $v$, and $M_i^j$
is element
$(i,j)$ of a matrix $M$.
Similarly, given a set $I\subseteq\{1,\ldots,J\}$, we
write $M_I$ and $M^I$ for the matrices consisting of the rows and
columns of $M$, respectively,
with indices in $I$.
Throughout, $E$ stands for the identity matrix and we write $\delta
_i^j$ for $E_i^j$.
We use the symbol $'$ for transpose.
The norms $\|\cdot\|_1$ and $\|\cdot\|_2$ stand for entrywise
$1$-norm and $2$-norm,
respectively, and are used for both vectors and matrices.

Given a measure space $(S, {\mathcal S})$, a measurable vector-valued
function $h\dvtx S\to\mathbb{R}^J$ on $(S,{\mathcal S})$, and a vector of measures
$\nu=(\nu_1,\ldots,\nu_J)$ on $(S,{\mathcal S})$, we set
\[
\int h(x) \nu(dx) = \int h(x)\cdot\nu(dx),
\]
provided the right-hand side exists. We shall also employ this
notation when $h$ and $\nu$ are matrix-valued. That is, we write for
$h\dvtx S\to\mathbb M^{J\times J}$ and an $\mathbb M^{J\times J}$-valued
measure $\nu$ on $(S,{\mathcal S})$,
\[
\int h(x) \nu(dx) = \int\bigl\langle h(x), \nu(dx)\bigr\rangle_{\mathrm{HS}},
\]
where $\langle\cdot,\cdot\rangle_{\mathrm{HS}}$ is the
Hilbert--Schmidt inner
product on $\mathbb M^{J\times J}$ given by
\[
\langle M_1,M_2\rangle_{\mathrm{HS}} =
\operatorname{tr}\bigl(M_1'M_2\bigr).
\]
%
For a function $g\dvtx  \mathbb M^{J\times J} \to\mathbb{R}$, we define
$\nabla g\dvtx  \mathbb M^{J\times J}\to
\mathbb M^{J\times J}$ as the function for which element $(i,j)$ is given
by the directional derivative of $g$ in the direction of the matrix
with only zero entries except for
element $(i,j)$, where its entry is 1.
We also write, for $i=1,\ldots,J$, $F_i = \{(z,a)\in\mathbb{R}^J_+
\times
\mathbb M^{J\times J}\dvtx  z_i=0\}$, $F^a_i = \{(z,a)\in\mathbb{R}^J_+
\times
\mathbb M^{J\times J}\dvtx  a_i=0\}$. The space of functions
$f\dvtx \mathbb{R}_+^{J} \times\mathbb M_+ ^ { J \times J }\rightarrow
\mathbb{R}$ which are
twice continuously differentiable with bounded derivatives is
denoted by \mbox{$C_b^2(\mathbb{R}_+^{J}\times\mathbb M_+ ^ { J \times J })$}.

We write $\mathbb{D}_+^J$ for the space of $\mathbb{R}_+^{J}$-valued
functions on $\mathbb{R}_+$ which are
right-continuous on $\mathbb{R}_+$ with left limits in $(0,\infty)$.
The subset of continuous functions is written as $C^J$, and $C^J_+$ denotes
the set of nonnegative continuous functions.
Similarly, we write $\mathbb{D}^{J\times J}$ for the space of $\mathbb
M^{J\times J}$-valued right-continuous
functions on~$\mathbb{R}_+$ with left limits. The subset of $\mathbb
M_+^{J\times J}$-valued functions is denoted
by $\mathbb{D}_+^{J\times J}$.

\section{A motivating one-dimensional result}\label{seconedim}
Fix some $\theta<0$. For any $\epsilon\ge0$, we
let $Z^\epsilon$ be a one-dimensional reflected Brownian motion with drift
$\theta-\epsilon<0$ and variance $\sigma^2$. That is,
\[
Z^\epsilon(t)=X^\epsilon(t)+Y^\epsilon(t) \ge0,
\]
where $X^\epsilon$ is a Brownian motion with drift $\theta-\epsilon$
and variance
$\sigma^2$, and the regulating term $Y^\epsilon$ is given by
\[
Y^\epsilon(t)= \max \Bigl( \sup_{ 0 \le s \le t }
\bigl[-X^\epsilon (s)\bigr],0 \Bigr).
\]
Suppose the family $\{Z^\epsilon\dvtx \epsilon\ge0\}$ is coupled in the
sense that $X^\epsilon(t) = W(t) +(\theta-\epsilon) t$
for some driftless Brownian motion $W$. Write $Z\equiv Z^0$.

It follows from
Theorem~1.1 in \citet{Mandel-ramanan}
[see also Lemma~5.2 and equation~(5.7) in \citet{MandelbaumMassey}]
that, for each $t\ge0$, the limit
%
\begin{equation}
\label{eq1dAtpre} A(t) \equiv\lim_{ \epsilon\rightarrow0+} \frac{1}{\epsilon}
\bigl(Z(t) - Z^{\epsilon} (t)\bigr)
\end{equation}
exists. We also have the following explicit formula:
%
\begin{equation}
\label{eq1dAt} A(t)=t- B(t),
\end{equation}
where
\[
B(t) = \sup\bigl\{ s \in[0,t]\dvtx Z(s) = 0\bigr\}
\]
and $\sup\varnothing=0 $ by convention. In view of the definition of $A$
in (\ref{eq1dAtpre}), we call it the derivative process of $Z$.

We now relate these notions to sensitivity analysis.
Our investigations are motivated
by the following sequence of equalities: for any ``smooth'' function
(performance
measure) $\phi$, one could expect that
%
\begin{equation}
\label{eqinterchange} \frac{d}{{d{\epsilon}}}\mathbb{E} \bigl[\phi\bigl({Z^\epsilon}(
\infty )\bigr) \bigr] = \mathbb{E} \biggl[ \frac{d}{{d{\epsilon}}} \phi
\bigl({Z^\epsilon }(\infty)\bigr) \biggr] = 
\mathbb{E} \bigl[A(\infty) \phi'\bigl(Z(\infty)\bigr) \bigr].
\end{equation}
Thus, to study (infinitesimal) changes in the steady-state performance
measure under
infinitesimal changes in the drift $\theta$, one is led to investigating
the stationary distribution of $(Z,A)$ (assuming it exists).
We are able to justify the interchange of expectation and derivative
in the above equalities in the one-dimensional case (see below),
but a justification in the setting of general multidimensional
constrained diffusions
requires a different set of techniques and falls outside the scope of
this paper.

%
\begin{figure}

\includegraphics{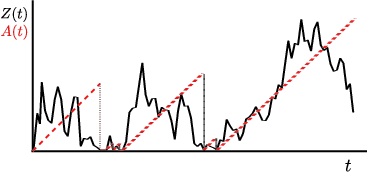}

\caption{Sample paths of $(Z,A)$ as a function of time.
The solid black curve is $Z$, while the dashed red curve is $A$.
The slope of $A$ is 1 whenever it is continuous,
and $A$ jumps to 0 whenever $Z$ hits 0.}\label{figZAonedim}
\end{figure}

One readily checks that the sample paths of the process $B$ are
nondecreasing, that they are right-continuous with left-hand limits
and that $A$ has positive drift and negative jumps. In particular,
the process $A$ is of finite variation, and $(Z, A)$ is a
semimartingale with jumps. An illustration of the process $(Z,A)$ is
given in Figure~\ref{figZAonedim}.
From Ito's formula in conjunction
with sample path properties of $A$,\vadjust{\goodbreak} we obtain the following result.
We suppress further details of the proof, since this program is carried
out in greater generality
in Section~\ref{secproofthm3}.

\begin{theorem}
\label{thmonedim}
Let $Z$ be a one-dimensional reflected Brownian
motion with drift $\theta$ and variance $\sigma^2$. Let $A$ be
defined in (\ref{eq1dAt}). Suppose that the process $(Z, A)$ has a unique
stationary distribution $\pi$. For any
$f \in C_b^2(\mathbb{R}_+\times\mathbb{R}_+)$, we have the following
relationship:
%
\begin{eqnarray}\label{eq1dbar}
0 &=& \int_0^\infty\!\!\int
_0^\infty  \biggl[ \frac{1}{2}{\sigma^2}\frac{{\partial^2}}{\partial{z^2}}f(z,a) + {\theta}\frac{\partial}{{\partial z}}f(z,a)\nonumber
\\
&&\hspace*{57pt}{} + \frac{\partial}{{\partial a}}f(z,a) - \frac{\partial}{{\partial a}}f(0,a) \biggr] \pi(dz,da)
\\
& &{} - \frac{\partial}{{\partial z}}f(0,0){\theta}.\nonumber
\end{eqnarray}
\end{theorem}

One can go further and derive the Laplace transform of $\pi$
using this theorem; see Appendix~\ref{apppiLaplace}.
One then finds that, for any $\alpha, \eta>0$,
%
\begin{equation}
\label{eq1dlap} \int_0^\infty\!\! \int
_0^\infty{{e^{ -
\alpha z - \eta a}}}  \pi(dz,da) =
\frac{{ - 2\theta}}{{\alpha
{\sigma^2} - \theta+ \sqrt{2\eta{\sigma^2} + {\theta^2}} }}.
\end{equation}
In particular, the theorem completely
determines the stationary measure $\pi$.
It is also possible to derive this result immediately from standard fluctuation
identities for Brownian motion with drift, using results from \citet{debickiquasiproduct2006}.
In fact, since the corresponding densities are known explicitly (or can
be found by inverting
the Laplace transform),
it is possible to write down the density of $(Z(\infty),A(\infty))$
in closed form.
Using the resulting expression, it can be verified directly that (\ref
{eqinterchange})
indeed holds.

\section{Oblique reflection maps and their directional derivatives}\label{secpreliminaries}
This section contains the technical preliminaries to formulate
a multidimensional analog of Theorem~\ref{thmonedim}.
We need the following definition to introduce the analogs of the
processes $A$ and $B$.

\begin{definition}[(Oblique reflection map)]
\label{defORM}
Suppose a given $J \times J$ real matrix $R$ can be written as $R=E-P$,
where $P$ is a nonnegative matrix with spectral radius less than one
and zeros on the diagonal.
Then for every $x \in\mathbb{D}^J$, there exists a unique pair
$(y,z)\in\mathbb{D}_{+}^{J} \times
\mathbb{D}_{+}^{J}$ satisfying the following conditions:

\begin{longlist}[(2)]
\item[(1)] $z(t)=x(t)+Ry(t) \ge0$ for $t\ge0$;
\item[(2)] $y(0)=0$, $y$ is componentwise nondecreasing and
\[
\int_0^\infty z(t)\,dy(t)=0.
\]
\end{longlist}
We write $y= \Phi(x)$ and $z=\Gamma(x)$ for the oblique reflection map.\vadjust{\goodbreak}
\end{definition}

The reflection map gives rise to left derivatives
as formalized in the following definition.
Existence of the derivatives is guaranteed by Theorem~1.1 in \citet{Mandel-ramanan}.

\begin{definition}[(Derivatives of the reflection map)]\label{defderivative}
Let $\chi(t)=tE$ and define
the $\mathbb M^{J\times J}$-valued functions $a$ and $b$ by defining
$a = \lim_{ \epsilon\rightarrow0+}a_{\epsilon}$ and
$b= \lim_{\epsilon\rightarrow0+} b_{\epsilon}$,
where the limits are to be understood as pointwise limits and, for
$j=1,\ldots,J$,
%
\begin{eqnarray}
\label{eqAi} \qquad a_{\epsilon}^j &\equiv& \frac{1}{\varepsilon} \bigl[
\Gamma(x) -\Gamma\bigl(x - \epsilon\chi ^j\bigr) \bigr],
\qquad
b_{\epsilon}^j \equiv -\frac{1}{\varepsilon} \bigl[
\Phi(x)-\Phi\bigl(x - \epsilon\chi^j\bigr) \bigr].
\end{eqnarray}
Then we have for each $t \ge0$,
%
\begin{equation}
\label{eqAt} a(t)= tE - R b(t).
\end{equation}
For notational convenience, we write $a=\Gamma'(x)$ and $b=-\Phi'(x)$.
\end{definition}

\section{Main results}\label{secmainresults}
This section states the main results of this paper. The first
result makes the connection between
derivatives and an augmented Skorohod problem, which we define
momentarily. The second result is a basic adjoint
relationship for the stationary distribution of solutions
to the augmented Skorohod problem with diffusion input.
The basic adjoint relationship is the analog of the equation $\pi' Q=0$
for Markov chains on a countable state space as mentioned in the \hyperref[secintroduction]{Introduction}.

\subsection{Augmented Skorohod problems and derivatives}\label{secaugsk}
In this section we introduce the augmented Skorohod problem and
connect it with derivatives of the oblique reflection map.

\begin{definition}[(Augmented Skorohod problem)]
\label{defaugskorohod} Suppose we are given two $J \times J$ real
matrices $R=E-P$ and $\widetilde R=E-\widetilde P $, where both $P$ and
$\widetilde P$ are nonnegative matrices with spectral radius less than
one and zeros on the diagonal. Given $(x, \chi) \in C^{J} \times
C^{J \times J}$ with $\chi$ componentwise nonnegative and
nondecreasing, we~say that $(z, y, a, b) \in C^J_+\times C^J_+
\times\mathbb{D}_+^{J\times J}\times\mathbb{D}_+^{J\times J}$
satisfies\vspace*{1pt} the
augmented Skorohod problem associated with $(R, \widetilde R)$ for $(x,
\chi)$ if the following conditions are satisfied:
\begin{longlist}[(3)]
\item[(1)] $z(t)=x(t)+Ry(t)$ for $t\ge0$;
\item[(2)] $y(0)=0$, $y$ is componentwise nondecreasing and
\[
\int_0^\infty z(t)\,dy(t)=0;
\]
\item[(3)] $a(t)=\chi(t) - \widetilde R b(t)$ for $t\ge0$;
\item[(4)] $b(0)=0$, $b(t) \ge0$, $b$ is componentwise nondecreasing and,
for $j=1,\ldots,J$,
%
\begin{equation}
\label{eqzdeltacomplem} \int_0^\infty z(t)\,d
b^j(t)=0;
\end{equation}
\item[(5)] For $i=1, \ldots, J$ and $t\ge0$, $z_i(t)=0$ implies $a_i(t)=0$.
\end{longlist}
\end{definition}

Building on results from \citet{Mandel-ramanan}, we show in
Appendix~\ref{appuniqueness} that the augmented Skorohod problem has
a unique solution.
To interpret solutions to the augmented Skorohod problem,
we found it easiest to think of the dynamics of $(z,a^j)$ for each
$j=1,\ldots,J$ separately.
When $z$ hits the face $z_I=0$, then $a^j$ jumps to the face $a^j_I=0$
in the direction of the unique vector in the column space of $\widetilde R^I$ which
brings it to that face.
We refer to Figure~\ref{figZAtwodim} for an illustrative example in
the two-dimensional case.

%
\begin{figure}

\includegraphics{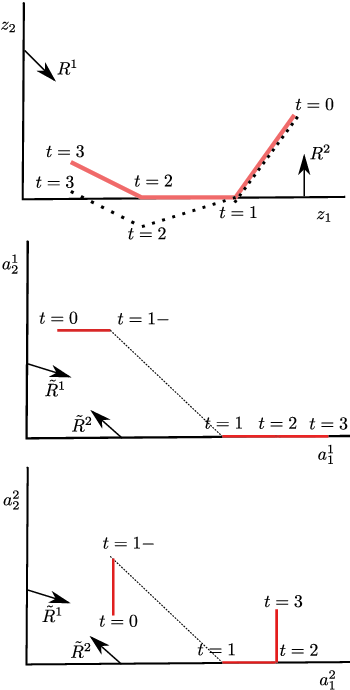}

\caption{The\vspace*{1pt} first diagram depicts a trajectory of $z$, with
corresponding ``free'' path $x$ (dotted).
In the second and third diagrams, the trajectories of $a^1$ and $a^2$
travel at unit rate right and up,
respectively, until $z$ hits $\partial\mathbb{R}^2_+$.
The face $z_2=0$ is hit at time $t=1$, causing $a^1$ and $a^2$ to jump
to the faces $a^1_2=0$ and $a^2_2=0$, respectively, in direction
$\widetilde R^2$.
Note that both $z(0)$ and $a(0)=\chi(0)$ are nonzero in these diagrams.}\label{figZAtwodim}
\end{figure}

Unlike requirements 2 and 4 in Definition~\ref{defaugskorohod},
requirement 5 is not a ``complementarity'' condition.
In view of the sample path dynamics in Figure~\ref{figZAtwodim},
it may seem reasonable to replace requirement 5 by $\int_0^\infty
a^j(t)\,dy(t)=0$ or another
complementarity condition between $(y,z)$ and $(a,b)$.
In that case, however, the augmented Skorohod will fail to have a
unique solution.
This can be seen by verifying that both the left derivative and the
right-derivative of the reflection map
satisfy $\int_0^\infty a^j(t)\,dy(t)=0$ but only the left derivative
(as defined in Definition~\ref{defaugskorohod}) satisfies requirement 5.

We now make a connection between derivatives (sensitivity analysis)
and solutions to the augmented Skorohod problem.
Note that, unlike in Figure~\ref{figZAtwodim}, one always has
$a(0)=\chi(0)=0$ in this case.

\begin{theorem} \label{thmasp}
Fix some $x\in C^J$, and let $z=\Gamma(x)$ and $y= \Phi(x)$ be
given by the oblique reflection map.
Define the derivatives $a=\Gamma'(x)$ and $b=-\Phi'(x)$ as in
Definition~\ref{defderivative}.
Set $\chi(t)= t E $ for $t\ge0$.
Then $(z, y, a, b)$ satisfies the augmented Skorohod problem associated
with $(R,R)$ for $(x, \chi)$.
\end{theorem}

\subsection{Stationary distribution of constrained diffusions and their derivatives}
%
Our second main result specializes to diffusion processes and studies
the stationary
distribution of solutions to the augmented Skorohod problem.
We show that it satisfies a generalized version of the basic adjoint
relationship (BAR) for reflected Brownian motion.
The proof relies on Ito's formula in conjunction with properties
developed in the previous section.
All results are formulated in terms of solutions to the augmented
Skorohod problem,
and the special case $\widetilde R =R$ is of primary interest for the
derivative process.


We first discuss the construction of constrained diffusion processes.
We work with a $d$-dimensional standard Brownian motion $W = \{W(t)\dvtx  t
\ge0\}$ adapted to some filtration $\{\mathcal F_t\}$,
on an underlying probability space $(\Omega,\mathcal F,\mathbb{P})$.
We are given functions $\theta$ and $\sigma$ on $\mathbb{R}_+^J$
taking values in $\mathbb{R}^J$ and $\mathbb M^{J\times d}$, respectively,
which satisfy the following standard Lipschitz and growth conditions:
(1)~For some \mbox{$L<\infty$}, we have
$\|\sigma(x) - \sigma(y)\|_2 + \|\theta(x) - \theta(y)\|_2 \le L \|
x - y\|_2$ for all $x, y \in\mathbb{R}_+^J$.
(2)~For some $K<\infty$, we have $\|\theta(x)\|_2^2+\|\sigma(x)\|
_2^2\le K(1+\|x\|_2^2)$ for $x\in\mathbb{R}_+^J$.
Given any initial condition $Z(0)$ with $\mathbb{E}\|Z(0)\|_2^2<\infty$,
there exists a pathwise unique, strong solution $\{Z(t)\dvtx t\ge0\}$ to the
stochastic differential equation with reflection~(SDER)
%
\begin{equation}
\label{eqSDER} dZ(t) = \theta\bigl(Z(t)\bigr) \,dt + \sigma\bigl(Z(t)\bigr) \,dW (t)
+ R \,dY (t).
\end{equation}
This equation is shorthand for the statement that, almost surely,
$Z=\Gamma(X)$ 
and $X(t) = Z(0) + \int_0^t \theta(Z(s)) \,ds + \int_0^t \sigma(Z(s))
\,dW(s)$ for $t\ge0$.
Moreover, $\mathbb{E}\|Z(t)\|_2^2$ is locally bounded as a function of $t$.
For these and related results, see \citet{MR0397893},
\citet{MR1110990}, \citet{karatzashrevebrownianmotion1991}, \citet{MR2261058}.
In particular, we have $Z(t)\in\mathbb{R}^J_+$ for all $t\ge0$.
We define the diffusion matrix $\Sigma$ through $\Sigma(z) = \sigma
(z)\sigma(z)'$ for $z\in\mathbb{R}_+^J$.
The special case of reflected Brownian motion follows upon taking
constant functions $\sigma$ and $\theta$.
Throughout this paper, we only work with constrained diffusion
processes that can be
obtained through the oblique reflection map of Definition~\ref{defORM},
and for which the time $Z$ spends $\partial\mathbb{R}_+^J$ has
Lebesgue measure zero almost surely
(this is only used in Section~\ref{appjumpmeasure}).
Although the notions of SDER and their solutions can be defined more generally,
our results cannot be extended to other settings using the present framework.

We next introduce an $\mathbb M^{J\times J}_+$-valued process $A=\{
A(t)\dvtx t\ge0\}$ through an augmented Skorohod problem.
Although the special choice $\widetilde R=R$ is most relevant for us given
the connection with the derivative process, our treatment
is not restricted to that case.
Given some $A(0)$, suppose that $(Z,Y,A,B)$ satisfies the augmented
Skorohod problem
associated with $(R,\widetilde R)$ for $(X,\chi)$ with $\chi(t) =A(0)+Et$
and $X$ as before.
Also suppose $(Z(0),A(0))$ has some distribution $u$ 
satisfying $\int\|z\|_2^2 u(dz,da)<\infty$. This assumption
guarantees existence of $Z$ on a sample-path level, and therefore we
do not need moment assumptions on $A(0)$ in order to guarantee
existence of the process $A$. The derivative process always starts
at the origin (i.e., the zero matrix), but here we have defined $A$
with an arbitrary initial distribution since we are interested in
stationary distributions for $(Z,A)$.
Recall that $\pi$ is said to be a stationary distribution for $(Z,A)$
if all marginal distributions of $(Z,A)$ are $\pi$
when $(Z(0),A(0))$ has distribution $\pi$, that is, for every bounded
measurable function
$f\dvtx  \mathbb{R}^J_+\times\mathbb M^{J\times J}\to\mathbb{R}$ and for
every $t\ge0$,
%
\begin{equation}
\label{eqdefstationary} \mathbb{E}\bigl[f\bigl(Z(t), A(t) \bigr)\bigr] = \int f(z,a)
\pi(dz,da).
\end{equation}
In view of Theorem~\ref{thmasp}, although a justification is outside
the scope of this paper, we think of the
stationary distribution of $(Z,A)$ with $\widetilde R=R$ as the
limiting distribution of $Z$ jointly with its derivative process.

We define the following operators:
$Q_I$ is a projection operator with the following property.
The matrix $Q_I(a)$ is obtained from $a$ by subtracting columns of
$\widetilde R^I$,
in such a way that the rows of $Q_I(a)$ with indices in $I$ become zero.
That is, we have
%
\begin{equation}
\label{eqdefQI} Q_I(a) = a - \widetilde R^I \bigl(\widetilde R_I^I\bigr)^{-1} a_I,
\end{equation}
where $\widetilde R_I^I$ is the principal submatrix of $\widetilde R$ obtained by
removing rows and columns from $\widetilde R$ which\vspace*{1pt} do not lie in $I$.
When $I = \varnothing$, we set $Q_I (a)=a$ for $a \in\mathbb M^{J\times J}$.

We also define operators $L$ and $T$ on $C^2_b(\mathbb{R}_+^J\times
\mathbb M_+^{J\times J})$ through
%
\begin{eqnarray}\label{eqdefT}
Lf(\cdot)&=&\tfrac{1}{2} \bigl\langle\Sigma(\cdot),
H_zf(\cdot) \bigr\rangle_{\mathrm{HS}} 
+\bigl\langle\theta(\cdot),\nabla_z f(\cdot)\bigr\rangle,
\nonumber\\[-8pt]\\[-8pt]
Tf(\cdot)&=&Lf(\cdot)+ \operatorname{tr}\bigl(
\nabla_a f (\cdot)\bigr),\nonumber
\end{eqnarray}
where $\nabla_z f$ and $H_zf$ denote the gradient and Hessian, respectively,
with respect to the first argument of $f$, and we use $\nabla_a f$ as
discussed in Section~\ref{secintroduction}.
Thus $\operatorname{tr}(\nabla_a)$ is shorthand for $\sum_{i=1}^J
\,d/{da_{ii}}$.

We can now formulate the following theorem, which is our second main result.
We write $I^c$ for the complement of a set $I$.
We write $z_I$ for the subvector of $z$ consisting of the components
with indices in $I$ as before,
and we also let $z|_I$ denote the projection of $z$ to $\{z\dvtx z_{I^c}=0\}$.

\begin{theorem}[(Basic adjoint relationship)]\label{thmbarprimenew}
Let the processes $Z$ and $A$ be defined as above,
and suppose that $(Z, A)$ has a unique stationary distribution $\pi$
with $\int(\|z\|_2^2+\|a\|_1)\pi(dz,da)<\infty$.
Then there exists a finite Borel measure $\nu$ on $\bigcup_i(F_i\cap F_i^a)$
and, for $I\subseteq\{1,\ldots,J\}$, finite Borel measures $u_I$ on
$(0,\infty)^{|I^c|}\times\mathbb M^{J\times J}_+$
such that for any $f \in
C^2_b(\mathbb{R}^J_+\times\mathbb M_+^{J\times J})$, the following
relationship
holds:
%
\begin{eqnarray}\label{eqbarprimenew}
\quad&& \int_{\mathbb{R}^J_+ \times\mathbb M_+^{J\times J} } T f (z,a) \,d\pi(z,a) +\int
_{\bigcup_i(F_i\cap F^a_i)} \bigl[R'\nabla_z f(z,a)\bigr]
\,d\nu(z,a)\nonumber
\\
&&\qquad{}+\sum_{I\subseteq\{1,\ldots,J\}\dvtx I\neq\varnothing} \int_{(0,\infty)^{|I^c|}\times\mathbb M_+^{J\times J}}
\bigl[f\bigl(z|_{I^c},Q_I(a)\bigr)
\\
&&\hspace*{179pt}{} -f(z|_{I^c},a)\bigr]\,du_I(z_{I^c},a)=0,\nonumber
\end{eqnarray}
where the operators $Q_I$ and $T$ are given in (\ref{eqdefQI}) and
(\ref{eqdefT}), respectively.
\end{theorem}

Section~\ref{secboundaryterm} shows that the measures $\nu$ and
$u_I$, $I\subseteq\{1,\ldots,J\}$
are completely determined by $\pi$, and expresses these measures in
terms of $\pi$.
We believe that~(\ref{eqbarprimenew}) fully determines $\pi$, $\nu$
and the $u_I$ measures,
but it is outside the scope of this paper to prove this.
For recent developments along these lines, see \citet{MR2834200},
\citet{kangramanancharacterization2012}.

Theorem~\ref{thmbarprimenew} does not have the same form as
Theorem~\ref{thmonedim}, and
our next result brings these two forms closer.
It is obtained by substituting a special class of functions in (\ref
{eqbarprimenew}) so that the last term in (\ref{eqbarprimenew}) vanishes.
To formulate the result, we need the following family of operators:
for any $f\in C^2_b(\mathbb{R}^J_+\times\mathbb M^{J\times J})$ and
each set $I \subseteq\{1,2, \ldots, J\}$, let
%
\begin{eqnarray}
(O_I f ) (z,a)&=&\sum_{S \subseteq{\{1,\ldots,J \}\setminus I} }
(-1)^{|S|} f\bigl(\Pi_{S \cup I } z, Q_I (a)\bigr),
\label{eqOIf}
\\
O&=& \sum_{ I \subseteq{\{1,\ldots,J \}} } O_I, \label{eqO}
\end{eqnarray}
where $\Pi_{S \cup I }$ is the projection operator which sets the
coordinates in ${S \cup I }$ equal to~0.

\begin{corollary} \label{corbarprime}
Let the processes $Z$ and $A$ be defined as above,
and suppose that $(Z, A)$ has a unique stationary distribution $\pi$
with $\int(\|z\|_2^2+\|a\|_1)\pi(dz,da)<\infty$.
Then there exists a finite Borel measure $\nu$ such that for any $f
\in
C^2_b(\mathbb{R}^J_+\times\mathbb M_+^{J\times J})$, the following
relationship
holds:
%
\begin{eqnarray}\label{eqbarprime}
&& \int_{\mathbb{R}^J_+ \times\mathbb M_+^{J\times J} } [ T \circ O f] (z,a) \,d\pi(z,a)
\nonumber\\[-8pt]\\[-8pt]
&&\qquad {} +\int_{\bigcup_i(F_i\cap F^a_i)} \bigl[R'\nabla_z (Of) (z,a)
\bigr] \,d\nu(z,a) =0,\nonumber
\end{eqnarray}
where the operators $T$ and $O$ are given in (\ref{eqdefT}) and (\ref{eqO}).
\end{corollary}

We remark that the proof of this corollary shows that (\ref
{eqbarprime}) is equivalent
to several equations. That is, for any $f \in
C^2_b(\mathbb{R}^J_+\times\mathbb M_+^{J\times J})$ and each set $I
\subseteq\{1,2,\ldots, J\}$, $\pi$ and $\nu$ must satisfy
%
\begin{eqnarray}\label{eqbarprimedisag}
&& \int_{\mathbb{R}^J_+ \times\mathbb M_+^{J\times J} } [T \circ O_I f ] (z,a) \,d
\pi(z,a)
\nonumber\\[-8pt]\\[-8pt]
&&\qquad{} +\int_{\bigcup_i(F_i\cap F^a_i)} \bigl[R'
\nabla_z (O_If) (z,a)\bigr] \,d\nu(z,a) =0,\nonumber
\end{eqnarray}
where the operators $O_I$ are defined in (\ref{eqOIf}). Note that
(\ref{eqbarprimedisag}) produces $2^J$ equations, one of which is trivial.
We refer to (\ref{eqbarprimedisag}) as BAR$_I$.

We first check that (\ref{eqbarprime}) yields the classical BAR for
the stationary distribution of
the reflected Brownian motion $Z$ when choosing
$f(z,a) \equiv g(z)$ for some smooth $g$. One readily checks that in
this case,
\[
(Of) (z,a)= \sum_{I \subseteq\{1,2, \ldots, J\}} \sum
_{S \subseteq{\{1,\ldots,J \}\setminus I} } (-1)^{|S|} g(\Pi _{S \cup I } z)= g(z).
\]
Substituting the above equation in (\ref{eqbarprime}), we
immediately obtain the well-known basic adjoint relationship
as introduced in \citet{harwil87a} for reflected Brownian motion,
%
\begin{equation}
\label{eqbar} \int_{\mathbb{R}^J_+ } L g (z) \,d\pi(z) + \int
_{\bigcup_iF_i} \bigl[R'\nabla_z g (z)\bigr]
\,d\nu(z) =0,
\end{equation}
where $ d\pi(z) = \int_{a \in\mathbb M^{J \times J}} \,d\pi(z,a)$ is the
stationary distribution for $Z$ and the Borel measure $d \nu(z)$ is
given by $d\nu(z)= \int_{a \in\mathbb M^{J \times J}} \,d\nu(z,a)$.

We next specialize (\ref{eqbarprime}) to the one-dimensional case,
and we verify that we recover Theorem~\ref{thmonedim}.
This shows in particular that (\ref{eqbarprime}) fully determines
$\pi$ if $J=1$.
Indeed, it is readily seen that
\[
(Of) (z,a)= (O_{\varnothing}f) (z,a) + (O_{\{1\}}f ) (z,a) = f(z, a)-
f(0,a)+f(0,0).
\]
Combining this with (\ref{eqbarprime}) gives (\ref{eq1dbar}),
but with $-\partial/\partial z f(0,0)\theta$ replaced with
$c \partial/\partial z f(0,0)$ for some constant $c=\nu(\{0,0\})>0$.
One can further show that $c=-\theta$, but we suppress the argument.

We next argue that none
of the $2^J-1$ nontrivial equations in (\ref{eqbarprimedisag}) can be dropped,
but we leave open the question whether they characterize $\pi$.
We do so by illustrating the interplay between the different BAR$_I$ in
a simple example.
Let $J=3$ and consider $Z=(Z_1,Z_2,Z_3)$, where $Z_1$, $Z_2$, and $Z_3$
are three independent one-dimensional
standard reflected Brownian motions.
We do not need the second argument $A$, and therefore we make no
distinction between (\ref{eqbarprimedisag})
and a ``classical'' analog of BAR$_I$ in (\ref{eqbarprimedisag}).
This classical analog is obtained by considering (\ref
{eqbarprimedisag}) for $f$ that do not depend on the second argument $a$;
cf.~how (\ref{eqbar}) was obtained from (\ref{eqbarprime}).
The process $Z$ has a unique stationary distribution $\pi$, which is a
product form; see, for example, \citet{harrisonwilliamsexponential1987} for details.
BAR$_{\{1,2\}}$ is equivalent with the third marginal distribution of
$\pi$ being exponential, with similar conclusions for
BAR$_{\{1,3\}}$ and BAR$_{\{2,3\}}$.
On the other hand, BAR$_\varnothing$ and BAR$_{\{j\}}$ for any $j\in\{
1,2,3\}$ contain no information on the marginal distributions, in the sense
that $O_\varnothing g = 0$ and $O_{\{j\}} g =0$ for functions of the form
$g(z) = f_1(z_1) +f_2(z_2)+f_3(z_3)$ (assuming appropriate smoothness).
Still, BAR$_{\{1\}}$ with BAR$_{\{1,2\}}$ and BAR$_{\{1,3\}}$ together
imply that the push-forward of $\pi$
under the projection map onto the last two coordinates has a product
form solution
since the two-dimensional reflected Brownian motion $(Z_2,Z_3)$
satisfies the so-called skew-symmetry condition;
see \citet{harrisonwilliamsexponential1987}, Theorem~6.1, and \citet{williamsskew1987}, Theorem~1.2.
Consequently, one can think of BAR$_{\{1\}}$ as describing the
dependencies between
the second and third components of $\pi$, with marginal distributions
determined
by BAR$_{\{1,2\}}$ and BAR$_{\{1,3\}}$, respectively.
Similarly, BAR$_\varnothing$ describes the dependencies of the three
two-dimensional push-forward measures of $\pi$.

\section{Characteristics of derivatives and proof of Theorem~\texorpdfstring{\protect\ref{thmasp}}{2}}\label{secproofthm2}
In this section, we prove Theorem~\ref{thmasp}.
We also collect additional sample path properties of derivatives,
with an emphasis on their jump behavior.
These properties will be used in the proof of Theorem~\ref{thmbarprimenew}.

Throughout this section, we work under the conditions of Theorem~\ref{thmasp}.
That is, we assume that $x\in C^J$ is given, and we write
$z=\Gamma(x)$, $y= \Phi(x)$, $a=\Gamma'(x)$ and $b=-\Phi'(x)$.
We also set $\chi(t)= t E $ for $t\ge0$.

\subsection{Complementarity}
This section connects the augmented Skorohod problem associated
with $(R,R)$ for $(x,\chi)$ with $(z,a)$.
Note that, in view of Definitions~\ref{defORM} and \ref{defderivative},
the first two requirements of the augmented Skorohod problem in
Definition~\ref{defaugskorohod}
are immediately satisfied for $(x,y,z)$.
It\vspace*{1pt} is immediate that $a=\chi- R b$ by definition of $a$, so we must
indeed choose $\widetilde R=R$.
We proceed with showing that $a$ and $b$ lie in $\mathbb{D}^{J\times
J}_+$ as required for
the augmented Skorohod problem,
but it is convenient to first establish part of the fourth requirement.

\begin{lemma} \label{lemgammaimonot}
The $\mathbb M^{J\times J}$-valued function $b$ is componentwise
nonnegative and nondecreasing.
\end{lemma}
\begin{pf}
Since $\chi(t)=tE$ for $t \ge0$, $\chi$ is evidently
nonnegative and nondecreasing. The monotonicity result in
Theorem~6 of \citet{kellawhitt96} shows that for any fixed $\epsilon
>0$, each
component of $b_{\epsilon}$ is nonnegative and nondecreasing.
The lemma follows from the fact that $b$ is the pointwise
limit of the sequences $\{b_{\epsilon}\}$ as $\epsilon
\rightarrow0+$.
\end{pf}

\begin{lemma} \label{lemAinD}
The $\mathbb M^{J\times J}$-valued functions $a$ and $b$ lie in
$\mathbb{D}_+^{J\times J}$.\vadjust{\goodbreak}
\end{lemma}
\begin{pf}
Since $b$ is nonnegative in view of Lemma~\ref{lemgammaimonot}, we
will have shown the claim for $b$ if we
verify that $b \in\mathbb{D}^{J\times J}$. We deduce from Theorem~1.1 in
\citet{Mandel-ramanan} that each component of $b$ is
upper semicontinuous and that it has left and right limits everywhere.
Since $b$ is nondecreasing by Lemma~\ref{lemgammaimonot},
these properties imply that $b\in\mathbb{D}^{J\times J}$.

We next show that $a \in\mathbb{D}_+^{J\times J}$. Clearly, since $b
\in\mathbb{D}_+^{J\times J}$,
we only need to show that $a$ is nonnegative.
Again by the monotonicity result in Theorem~6 of
\citet{kellawhitt96}, for any fixed $\epsilon>0$, each
component of $a_{\epsilon}$ is nonnegative.
This completes the proof of the lemma after letting $\epsilon\to0+$.
\end{pf}

We next investigate the fourth and fifth requirement
of Definition~\ref{defaugskorohod}.
To this end, we need a characterization of $b$
which relies heavily on \citet{Mandel-ramanan}.

\begin{lemma}
$b$ is the unique solution to the following system of equations: for
$i,j=1,\ldots,J$
and $t\ge0$,
\[
b^j_i(t) = \sup_{s \in{\Phi_{(i)}}(t)} \bigl[
\delta_{i}^j s + {\bigl[P'b^j
\bigr]_i}(s)\bigr],
\]
%
where the supremum over an empty set should be interpreted as zero
and 
%
\begin{eqnarray}
{\Phi_{(i)}}(t) &=&\bigl\{ s \in[0,t]\dvtx z_i(s) =
0\bigr\}. \label{eqPhij}
\end{eqnarray}
\end{lemma}
\begin{pf}
We use Theorem~1.1 of \citet{Mandel-ramanan}, which can be simplified
in view of
Lemma~\ref{lemgammaimonot} and the nonnegativity of the matrix $P$.
This theorem states that
%
\begin{equation}
\label{eqgamma} b^j_i(t) = \cases{ 0, &\quad if $t
\in(0,t_{(i)})$,
\cr
\displaystyle \sup_{s \in{\Psi_{(i)}}(t)} \bigl[
\delta_{i}^j s + {\bigl[P'b^j
\bigr]_i}(s)\bigr], &\quad if $t \in[t_{(i)},\infty)$,}
\end{equation}
where
$t_{(i)}= \inf\{ t \ge0\dvtx {z_i}(t) = 0\}$
and
$\Psi_{(i)}(t) =\{ s \in[0,t]\dvtx z_i(s) = 0,y_i (s) = y_i(t)\}$.
Observe that, again using Lemma~\ref{lemgammaimonot},
the supremum must be attained at the
rightmost end of the closed interval $\Psi_{(i)}(t)$.
Since $y$ is nondecreasing and $\int_{0}^{t} z_i(s) \,dy_i(s)=0$, this
is also the rightmost point of the closed set $\Phi_{(i)}(t)$.
This establishes the lemma in view of the convention used for the
supremum of an empty set.
\end{pf}

\begin{lemma} \label{lemZgammacomp}
Fix any $j=1, \ldots, J$, and we have
%
\begin{equation}
\label{eqZgamma} \int_0^\infty z(t) \,d
b^j(t)=0.
\end{equation}
\begin{pf}
Fix some $i=1,\ldots,J$.
Note that if $z_i(t)>0$ at time $t$, we deduce from the path
continuity of $z$ that there exists some $\epsilon>0$ such that
$z_i(s)>0$ for $s \in(t-\epsilon, t+ \epsilon)$. This implies that
$\Phi_{(i)}(s)$ is constant as a set-valued function for $s \in
(t-\epsilon, t+\epsilon)$. Thus $b_i(s)$ is constant for $s
\in(t-\epsilon, t+\epsilon)$ by (\ref{eqgamma}). Since $i$ is
arbitrary, this yields (\ref{eqZgamma}).
\end{pf}
\end{lemma}

\begin{lemma} \label{lemboundary}
If $z_i(t)=0$ for some $i$, then we have $a_i (t)=0$.
\end{lemma}
\begin{pf}
Suppose $z_i(t)=0$.
In view of Lemma~\ref{lemgammaimonot}, we deduce from (\ref
{eqgamma}) that, for any $j=1,\ldots,J$,
\[
b^j_i(t) = \delta_i^j t +
\bigl[P'b^j\bigr]_i(t).
\]
Now it follows from (\ref{eqAt}) and $R=E-P'$ that
\[
a^j_i (t)= \delta_i^j t-
\bigl[Rb^j\bigr]_i(t) = \delta_i^j
t-b^j_i(t) + \bigl[P'b^j
\bigr]_i(t) =0,
\]
which completes the proof of the lemma.
\end{pf}

The above two lemmas together with Lemma~\ref{lemAinD} yield two
further complementarity conditions.

\begin{corollary}
For any $j=1, \ldots, J$, we have
%
\begin{eqnarray}\label{eqAy}
\int^\infty a^j (t) \,dy(t)&=&0,
\qquad
\int_0^\infty a^j (t) \,d
b^j(t)=0.
\end{eqnarray}
\end{corollary}

\begin{pf*}{Proof of Theorem \ref{thmasp}}
The claim is now immediate from (\ref{eqAi}) in conjunction with
Lemmas~\ref{lemgammaimonot}, \ref{lemAinD}, \ref{lemZgammacomp}
and \ref{lemboundary}.
\end{pf*}

\subsection{Jumps of $a$}
In this section, we collect sample path properties of $a$
related to its jump behavior.
This plays a critical role in the derivation of Theorem~\ref{thmbarprimenew} and Corollary~\ref{corbarprime}.

The next lemma states that $a$ is linear whenever $z$ is in the interior of~$\mathbb{R}_+^J$.

\begin{lemma}
\label{lemjumptime} If $z(t) \in\mathbb{R}_{+}^J \setminus
\partial
\mathbb{R}_{+}^J$ for $t \in[\alpha,\beta]$, then we have for $t
\in[\alpha,\beta]$
\[
a(t)= a(\alpha) + (t-\alpha) E.
\]
In particular, $a$ is continuous on $(\alpha,\beta)$ and can only have
jumps when \mbox{$z \in\partial\mathbb{R}_{+}^J$}.
\end{lemma}

\begin{pf}
In view of (\ref{eqAt}), it suffices to show that $b$ is
constant for $t \in[\alpha,\beta]$. Since $z(t) \in\mathbb{R}_{+}^J
\setminus\partial\mathbb{R}_{+}^J$ for $t \in[\alpha,\beta]$, we
obtain from~(\ref{eqPhij}) that for each $i=1, \ldots, J$, $\Phi_{(i)}(t)$ is
constant as a set-valued function. Therefore, we deduce from~(\ref{eqgamma}) that $b(t) $ is a constant in $\mathbb M^{J \times
J}$ for $t \in[\alpha,\beta]$. The proof of the lemma is complete.
\end{pf}

For any function $g$ on $\mathbb{R}_+$, we write $\Delta g(t)=g(t)-g(t-)$.
In view of the above lemma, we can characterize the continuous part
of the function $a$. Formally, we write
\[
a(t)=a^c(t)+a^d(t),
\]
where
\[
a^d(t)= \sum_{s \le t} \Delta a(s).
\]
We have the following corollary.

\begin{corollary} \label{corAct}
$a^c(t)=a(0)+tE$ for any $t \ge0$.
\end{corollary}

We next characterize the jump direction of $a$ when a jump occurs.
%
\begin{lemma}
\label{lemjumpdir}
Fix a nonempty set $I \subseteq\{1, 2, \ldots, J\}$ and some $t> 0$.
Suppose that $z_k (t)=0$ for $k \in I$ and
$z_i(t)>0$ for $i \notin I$. If $\Delta a(t)\neq0$, then we must have
\[
\Delta a(t) = -\sum_{k \in I} R^k [\Delta
b]_k(t).
\]
\end{lemma}
\begin{pf}
Since $z_i(t)>0$ for $i \notin I$, we deduce from the sample path
continuity of $z$ that there exists some $\epsilon>0$ such that for
$i \notin I$, $z_i(s)>0$ for $s \in(t-\epsilon, t]$.
This yields that for $i \notin I$, $\Phi_{(i)}(s)$ is a constant as a
set-valued
function for $s \in(t-\epsilon, t]$. From (\ref{eqgamma}) we infer
that for $i \notin I$, $b_i(s)$ is constant for $s \in
(t-\epsilon, t]$. This implies that $[\Delta b]_i(t)=0$
for $i \notin I$, and therefore that
\[
\Delta a(t)= -R \Delta b(t) = -\sum_{k=1}^J
R^k [\Delta b]_k(t) =-\sum
_{k\in I} R^k [\Delta b]_k(t).
\]
This completes the proof of the lemma.
\end{pf}


\section{A basic adjoint relationship and proof of Theorem~\texorpdfstring{\protect\ref{thmbarprimenew}}{3}}\label{secproofthm3}
This section is devoted to the proof of Theorem~\ref{thmbarprimenew}
and Corollary~\ref{corbarprime}. The key
idea is to apply Ito's formula to the semimartingale $(Z,A)$ and use
sample path properties of $(Z,A)$ to analyze the stationary measure.
This is a standard approach in the context of reflected Brownian
motion, but
the analysis here exposes new features due to the presence of jumps in
the process $A$.
Throughout, we work with the augmented filtration generated by $W$ and
$(Z(0),A(0))$.

\subsection{Ito's formula for the semimartingale $(Z,A)$}
In this section, we apply Ito's formula to the semimartingale $(Z,A)$.
We first show that $(Z,A)$ is a semimartingale, that is, each of its
components is a semimartingale. Recall that a semimartingale is
an adapted process which is the sum of\vadjust{\goodbreak} a local martingale and
a
finite variation process, with sample paths in $\mathbb{D}$.
For more detail, we refer readers to \citet{protterbook}, Chapter~3, or
\citet{jacodshiryaevlimit2003}, Chapter~1.
%
\begin{lemma}
$(Z,A)$ is a semimartingale.
\end{lemma}
\begin{pf}
The process $(Z, A)$ is adapted. This is a well-known property of $Z$,
and $A(t)$ is a deterministic functional of
$\{Z(s)\dvtx 0\le s\le t\}$ and $A(0)$ since it arises from an augmented
Skorohod problem.
We know from Lemma \ref{lemAinD} that each component of the
process $(Z,A)$ lies in $\mathbb{D}$. Since $Z$ is a semimartingale,
to show
$(Z,A)$ is a semimartingale, it suffices to show that $A$ is a
semimartingale.
In fact, from Lemma~\ref{lemgammaimonot} and (\ref{eqAt}) we immediately
deduce that $A$ is a finite variation process, that is, the paths of
$A$ are almost surely of finite
variation on $[0, T]$ for any $T>0$. In particular, $A$ is a semimartingale.
\end{pf}

By Ito's formula, for example, \citet{jacodshiryaevlimit2003}, Section~I.4,
we deduce from (\ref{eqSDER}) that for any
$f\in C_b^2(\mathbb{R}^J_+\times\mathbb M^{J\times J})$, we have
%
\begin{eqnarray}\label{eqito}\label{ito}
f\bigl({Z}(t),A(t)\bigr) &=& f\bigl({Z }(0),{A}(0)\bigr)\nonumber
\\
&&{}  + \int _0^t \bigl[\sigma\bigl(Z(s)\bigr)'{
\nabla_z} f\bigl({Z}(s),A(s-)\bigr) \bigr] \,dW(s)\nonumber
\\
&&{}+ \int_0^t \bigl[R'
\nabla_z f\bigl(Z(s),A(s-)\bigr)\bigr] \,dY(s)
\nonumber\\[-8pt]\\[-8pt]
&&{} +\int_0^t Lf\bigl(Z(s),A(s-)\bigr) \,ds\nonumber
\\
&&{} + \int _0^t {\nabla_a} f\bigl({Z}(s),A(s-)
\bigr)\,dA^c(s)
\nonumber
\\
&&{}+ \sum_{s \le t} \bigl[f\bigl({Z}(s),A(s)\bigr) - f
\bigl({Z}(s),A(s-)\bigr)\bigr]. \nonumber
\end{eqnarray}
Compared to the formulation in
Theorem I.4.57 of \citet{jacodshiryaevlimit2003}, we have absorbed
the last sum of the jump part into the integral
\[
\int_0^t {\nabla_a}
f\bigl({Z}(s),A(s-)\bigr)\,dA^c(s).
\]
This is justified by noting that,
since $\Delta A(s) = -\widetilde R\Delta B(s)$ for some nonnegative and
(componentwise)
nondecreasing process $B$ according to Definition~\ref{defaugskorohod},
%
\begin{equation}
\label{eqboundjumpsA} \sum_{s\le t} \bigl\|\Delta A(s)
\bigr\|_1 \le C \sum_{s\le t} \bigl\|\Delta B(s)
\bigr\|_1 = C \bigl\|B(t)\bigr\|_1<\infty,
\end{equation}
where $C$ denotes some constant depending on $\widetilde R$.
Note that this also implies that the last term on the right-hand side
of (\ref{eqito}) is absolutely convergent.
Indeed, combining the above bound with $f\in C_b^2(\mathbb
{R}^J_+\times\mathbb M^{J\times J})$
yields $\sum_{s \le t} |f({Z}(s),A(s)) - f({Z}(s),A(s-))|<\infty$.

Suppose that $(Z,A)$ is positive
recurrent and has a unique stationary distribution~$\pi$.
Henceforth we assume that $(Z(0),A(0))$ has distribution $\pi$,
and we write $\mathbb{E}_\pi$ instead of $\mathbb{E}$. After taking
an expectation with respect to $\pi$ on both sides of (\ref{eqito}),
the term involving $dW$ vanishes since
it is a martingale term.
We next analyze the second to last term on the right-hand side.
From Corollary \ref{corAct} and the fact that $A$ has countably many jumps
(Lemma~\ref{lemgammaimonot}), we deduce that
\begin{eqnarray*}
\mathbb{E}_{\pi} \int_0^t {
\nabla_a} f\bigl({Z}(s),{A}(s-)\bigr) \,dA^c(s) &=&
\mathbb{E}_{\pi} \int_0^t {
\nabla_a} f\bigl({Z}(s),{A}(s-)\bigr) \,d (s E )
\\
& = & \mathbb{E}_{\pi} \int_0^t {
\nabla_a} f\bigl({Z}(s),{A}(s)\bigr) \,d (s E )
\\
& = &\mathbb{E}_{\pi} \int_0^t
\operatorname{tr}\bigl({\nabla_a} f\bigl({Z}(s),{A}(s)\bigr)\bigr) \,d
s.
\end{eqnarray*}
Since $ f\in C^2_b(\mathbb{R}^J_+\times\mathbb M^{J\times J})$, we
have from
Fubini's theorem and the definition of stationarity in (\ref
{eqdefstationary}) that
\begin{eqnarray*}
\mathbb{E}_{\pi} \int_0^t
\operatorname{tr}\bigl({\nabla_a} f\bigl({Z}(s),{A}(s)\bigr)\bigr) \,d s
&=& \int_0^t \mathbb{E}_{\pi}
\operatorname{tr}\bigl({\nabla_a} f\bigl({Z}(s),{A}(s)\bigr)\bigr) \,d s
\\
&=& t \int\operatorname{tr}\bigl(\nabla_a f(z,a)\bigr) \,d \pi(z,a).
\end{eqnarray*}
Thus we obtain
\[
{\mathbb{E}_{\pi}}\int_0^t {
\nabla_a} f\bigl({Z}(s),{A}(s-)\bigr)\,dA^c(s) = t \int
\operatorname{tr}\bigl(\nabla_a f(z,a)\bigr) \,d \pi(z,a).
\]
A similar argument applies to the fourth term on the right-hand side of
(\ref{eqito}).
We conclude that, for each $t \ge0$ and each $f\in C^2_b(\mathbb
{R}^J_+\times\mathbb M^{J\times J})$,
%
\begin{eqnarray}
\label{eqitopre}
\qquad 0 &=& t \int{ \bigl[ Tf(z,a) \bigr]} \,d\pi(z,a) +
\mathbb{E}_{\pi} \int_0^t
\bigl[R' \nabla_z f\bigl(Z(s),A(s-)\bigr)\bigr] \,dY(s)
\nonumber\\[-8pt]\\[-8pt]
&&{} + {\mathbb{E}_\pi} \sum_{s \le t} {
\bigl[f\bigl({Z}(s),A(s)\bigr) - f\bigl({Z}(s),A(s-)\bigr)\bigr]},\nonumber
\end{eqnarray}
where $T$ is given in (\ref{eqdefT}).
This equation serves as the starting point for proving Theorem~\ref{thmbarprimenew}.

\subsection{The boundary term}\label{secboundaryterm}
In this section we rewrite the boundary term in
(\ref{eqitopre}), that is, the term involving $dY$. Let $\nu=(\nu
_1,\ldots,\nu_J)$ be the unique
vector of measures on $\partial\mathbb{R}^J_+\times\mathbb
M^{J\times J}$ for
which
\[
\int h(z,a) \nu(dz,da) = \mathbb{E}_\pi\int_0^1
h\bigl(Z(s),A(s-)\bigr)\,dY(s)
\]
for all continuous $h\dvtx \partial\mathbb{R}^J_+\times\mathbb M^{J\times
J}\rightarrow\mathbb{R}^J$ with compact support.
This is a well-defined measure by the following lemma.
For a different proof in the reflected Brownian motion case, see \citet{harwil87a}, Section~8.
%
\begin{lemma}
We have $\mathbb{E}_\pi Y(1)<\infty$ componentwise.
\end{lemma}
\begin{pf}
Since $Y(1)\ge0$, it is enough to show that $\mathbb{E}_\pi\|RY(1)\|
_1<\infty$.
We prove the stronger statement that $\mathbb{E}_\pi\| RY(1)\|
_2^2<\infty$.
From the fact that $Z$ satisfies the SDER (\ref{eqSDER}), we obtain
\[
\mathbb{E}_\pi\bigl\llVert R Y(1)\bigr\rrVert _2^2
= \mathbb{E}_\pi\biggl\llVert Z(1) - Z(0) - \int_0^1
\theta\bigl(Z(s)\bigr)\,ds - \int_0^1 \sigma
\bigl(Z(s)\bigr) \,dW(s)\biggr\rrVert _2^2.
\]
It follows from the fact that $t\mapsto\mathbb{E}_\pi\|Z(t)\|_2^2$
is locally bounded and
the growth condition on $\theta$ that $\mathbb{E}_\pi\llVert \int_0^1\theta(Z(s))\,ds\rrVert _2^2<\infty$.
Similarly, we have
\[
\mathbb{E}_\pi\biggl\llVert \int_0^1
\sigma\bigl(Z(s)\bigr) \,dW(s)\biggr\rrVert _2^2 =
\mathbb{E}_\pi\int_0^1
\operatorname{tr}\Sigma\bigl(Z(s)\bigr) \,ds <\infty,
\]
where the finiteness follows from the growth condition on $\sigma$.
\end{pf}

Our next goal is to give a characterization of measure $\nu$ in
terms of $\pi$, which we carry out through Laplace transforms. We
start with
determining the support of $\nu$.

\begin{lemma}
The support of $\nu$ is $\bigcup_i(F_i\cap F^a_i)$.
\end{lemma}

\begin{pf}
In view of Lemma~\ref{lemAinD}, it is clear that $A$ can have at most
countably many jumps.
For any continuous $h\dvtx \partial\mathbb{R}^J_+\times\mathbb M^{J\times
J}\rightarrow\mathbb{R}^J$ with compact
support, we have
\[
\int_0^1 h\bigl(Z(s),A(s-)\bigr)\,dY(s)= \int
_0^1 h\bigl(Z(s),A(s)\bigr)\,dY(s),
\]
since the measure $dY$ is continuous
and the integrand has countably many jumps by Lemma~\ref{lemgammaimonot}.
It follows from the definition of $\nu$ that
%
\begin{equation}
\label{eqboundaryterm} \int h(z,a) \nu(dz,da) = \mathbb{E}_\pi\int
_0^1 h\bigl(Z(s),A(s)\bigr)\,dY(s).
\end{equation}
The complementarity conditions $\int_0^\infty Z(t) \,dY(t)=0$ and (\ref
{eqAy}) imply the lemma.\vadjust{\goodbreak}
\end{pf}

On combining equations~(\ref{eqitopre}) and (\ref{eqboundaryterm})
we obtain that for any $f \in\break  C^2_b(\mathbb{R}^J_+\times\mathbb
M^{J\times J}_+)$,
%
\begin{eqnarray}
\label{eqbarprimepre}
\qquad 0&=&\int_{\mathbb{R}^J_+ \times\mathbb M_+^{J\times J}} \bigl[ T f(z,a)\bigr] \,d\pi(z,a)
+\int_{\bigcup_i(F_i\cap F^a_i)} \bigl[R'\nabla_z f(z,a)
\bigr] \,d\nu(z,a)
\nonumber\\[-8pt]\\[-8pt]
&&{}+ \frac{1}{t}{\mathbb{E}_\pi} \sum
_{s \le t} {\bigl[f\bigl({Z}(s),A(s)\bigr) - f\bigl({Z}(s),A(s-)
\bigr)\bigr]}.\nonumber
\end{eqnarray}
We now express the Laplace transform of $\nu$ in terms of the
Laplace transform of~$\pi$. Set $f(z,a) = \exp( - \eta'z -
\langle\alpha, a\rangle_{\mathrm{HS}}) \in C_b^2(\mathbb
{R}^J_+\times\mathbb M^{J\times J}_+) $ where
$(\eta, \alpha)\in\mathbb{R}^J_+\times\mathbb M^{J\times J}_+$.
After substituting
$f$ in (\ref{eqbarprimepre}), we obtain
%
\begin{equation}
\label{eqlaplacebarp} \pi^* ( \eta,\alpha) - \sum_{j=1}^J
\bigl(R'\eta\bigr)_j \nu_j^* ( \eta,\alpha)
+ H ( \eta,\alpha)=0,
\end{equation}
where
\begin{eqnarray*}
\pi^* ( \eta,\alpha)&=& \int_{\mathbb{R}^J_+ \times\mathbb
M^{J\times J}}
\Biggl[\frac{1}2 \eta'\Sigma(z)\eta+\eta'
\theta (z)-\sum_{i=1}^J
\alpha_i \Biggr]{{e^{ -
\eta\cdot z - \alpha\cdot a }}\,d\pi(z,a)},
\\
\nu_j^* ( \eta,\alpha)&=& \int_{F_j\cap F^a_j}
{{e^{ - \eta
\cdot z - \alpha\cdot a }}\,d\nu_j (z,a)},
\\
H( \eta,\alpha)&=&{\mathbb{E}_\pi} \sum
_{s \le1} {\bigl[e^{- \eta
\cdot
Z(s)} \cdot\bigl( e^{-\alpha\cdot A(s) }-e^{-\alpha\cdot A(s-)}
\bigr)\bigr]}.
\end{eqnarray*}
Dividing (\ref{eqlaplacebarp}) by $\eta_j>0$ and letting $\eta_j
\rightarrow\infty$, we deduce that
%
\begin{equation}
\label{eqlapnu} \nu_j^* ( \eta,\alpha)= \frac{1}2 \lim
_{\eta_j \to\infty} {\eta_j} \int_{\mathbb{R}^J_+ \times\mathbb M^{J\times J}}
\Sigma_{jj}(z) {{e^{ - \eta\cdot z - \alpha\cdot a }}\,d\pi(z,a)},
\end{equation}
where we have used the fact that $\nu_j(F_j \cap F_i)=0$ for $i \ne
j$ so that\break  $ \lim_{\eta_j \to\infty}
\nu_i^* (\eta,\alpha)=0$ by\vspace*{1pt} the dominated convergence theorem.
Since all terms in (\ref{eqlaplacebarp}) vanish in the limit by
dominated convergence except for
the term with $\nu^*_j$ and the term with $\pi^*$,
existence of the limit in (\ref{eqlapnu}) follows immediately from
the fact that $\nu_j(\eta,\alpha)$
does not depend on $\eta_j$.
Under further regularity conditions on $\pi$, one can use the initial
value theorem
for Laplace transforms to show that $d \nu_j= \frac{1}{2} \Sigma
_{jj} \,d \pi_j$
for an appropriate restriction $\pi_j$ of $\pi$.
Carrying out this procedure provides little additional insight, and we
therefore suppress further details.

\subsection{The jump term}
We now proceed investigating the jump term, that is, the term in
(\ref{eqitopre}) involving the countable sum.
Lemma~\ref{lemjumptime} implies that jumps in $A$ can only occur when
$Z$ lies hits the boundary $\partial\mathbb{R}^J_+$
of the nonnegative orthant, which motivates the following definition.
For $I\subseteq\{1,\ldots,J\},\break  I\neq\varnothing$, we define measures $u_I$
on $\mathbb{R}^{|I^c|}_+\times\mathbb M^{J\times J}_+$ with support
in $(0,\infty)^{|I^c|}\times\mathbb M^{J\times J}_+$.
We set, for Borel sets $G\subseteq(0,\infty)^{|I^c|}, C\subseteq
\mathbb M^{J\times J}_+$,
\[
u_I(G,C) = \mathbb{E}_\pi\sum
_{s\le1\dvtx  Z_I(s)=0, Z_{I^c}(s)\in G,
A(s)\neq A(s-)} 1_C \bigl\{ A(s-)\bigr\}.
\]
This is a well-defined $\sigma$-finite measure because of (\ref
{eqboundjumpsA}) and
$\mathbb{E}_\pi\|B(1)\|_1 = \mathbb{E}_\pi\|A(1)-A(0)-E\|_1\le
2\mathbb{E}_\pi\|A(0)\|_1+ J<\infty$,
so that
\[
\mathbb{E}_\pi\biggl|\sum_{s\le1}\bigl[f\bigl(Z(s),A(s)\bigr)-f\bigl(Z(s),A(s-)\bigr)\bigr]\biggr|<\infty
\]
for $f\in C_b^2(\mathbb
{R}^J_+\times\mathbb M^{J\times J})$.
It is possible to express these measures in terms of $\pi$
using the theory of distributions; this is done in Section~\ref{appjumpmeasure}.

The primary objective of this subsection is to show that the jump term
in (\ref{eqitopre}) vanishes
for a special class of functions, which is key in our proof of
Corollary~\ref{corbarprime}.
Throughout, we fix a set $I \subseteq\{1, \ldots, J\}$.
Recall the definition of $O_I$ in (\ref{eqOIf}). It is our aim to
show that
the jump term vanishes for functions of the form $O_I f$, where $f\in
C^2_b(\mathbb{R}^J_+\times\mathbb M^{J\times J})$
as before. We first introduce a lemma.

\begin{lemma} \label{lemgivanish}
For any $f\dvtx  \mathbb{R}^J_+ \times\mathbb M^{J\times J} \to\mathbb
{R}$, if $z_j=0$ for some $j \notin I$,
then for any $a \in\mathbb M^{J\times J}$ we have
\[
\sum_{S \subseteq{\{1,\ldots,J \}\setminus I} } (-1)^{|S|} f(\Pi_{S \cup I }
z, a) =0.
\]
In particular, if $z_j=0$ for some $j \notin I$, then we have $O_I f(z,a)=0$.
\end{lemma}
\begin{pf}
Suppose $z_j=0$ for some $j \notin I$. Then for any set $ S
{\subseteq\{1,\ldots,J\} \setminus I}$ with $j \notin S$, we have
$\Pi_{S \cup I} z= \Pi_{S \cup I \cup\{ j\} } z$.
Using this observation, we deduce that
\begin{eqnarray*}
&& \sum_{S \subseteq{\{1,\ldots,J \}\setminus I} } (-1)^{|S|} f(
\Pi_{S \cup I } z, a)
\\
&&\qquad = \sum_{S\subseteq\{1,\ldots,J\} \setminus I\dvtx j\in S} (-1)^{|S|} f(
\Pi_{S \cup I} z, a )
\\
&&\quad\qquad{}+ \sum_{S\subseteq\{1,\ldots,J\} \setminus I\dvtx  j\notin S}
(-1)^{|S|} f(\Pi_{S \cup I} z, a )
\\
&&\qquad =\sum_{S\subseteq\{1,\ldots,J\} \setminus I\dvtx  j\in S} (-1)^{|S|} f(
\Pi_{S\cup I} z, a )
\\
&&\quad\qquad{}+ \sum_{S\subseteq\{1,\ldots,J\}
\setminus I\dvtx  j\notin S }
(-1)^{|S|} f(\Pi_{S \cup I \cup\{ j\}} z, a )
\\
&&\qquad =\sum_{ S\subseteq\{1,\ldots,J\} \setminus I\dvtx j\in S} (-1)^{|S|} f(
\Pi_{S \cup I} z, a )
\\
&&\quad\qquad{}+ \sum_{\widetilde S\subseteq\{1,\ldots,J\}
\setminus I\dvtx  j\in\widetilde S}
(-1)^{|\widetilde S|-1} f(\Pi_{\widetilde S \cup
I } z, a )
\\
&&\qquad =0.
\end{eqnarray*}
The proof of the lemma is complete.
\end{pf}

Now we are ready to show that the jump term vanishes for functions of
the form
$O_I f$. For any $K \subseteq\{1, \ldots, J\}$, $Z_K$
denotes the process whose components are those of $Z$ with indices
in $K$.

\begin{lemma}
\label{lemnojumpterm}
For each $t \ge0$ and any measurable $f\dvtx  \mathbb{R}^J_+\times\mathbb
M^{J\times J}\to\mathbb{R}$, we have\vspace*{-2pt}
%
\begin{equation}
\label{eqnojumpterm} {\mathbb{E}_\pi} \sum_{s \le t}
{\bigl[O_I f\bigl({Z}(s),A(s)\bigr) - O_I f
\bigl({Z}(s),A(s-)\bigr)\bigr]}=0.
\end{equation}
\end{lemma}
\begin{pf}
By Lemmas~\ref{lemjumptime}~and~\ref{lemgivanish}, we have
\begin{eqnarray*}
&&  {\mathbb{E}_\pi} \sum_{s \le t} {
\bigl[O_I f\bigl({Z}(s),A(s)\bigr) - O_I f
\bigl({Z}(s),A(s-)\bigr)\bigr]}
\\[-1pt]
&&\qquad  = \sum_ { \varnothing\neq K \subseteq\{1,\ldots,J\}}{\mathbb {E}_\pi} \sum
_{s \le t\dvtx  Z_K(s)=0, Z_{\{1,\ldots,J\} \setminus
K}(s)>0} \bigl[O_I f\bigl({Z}(s),A(s)
\bigr)
\\[-3pt]
&&\hspace*{218.5pt}{}- O_If\bigl({Z}(s),A(s-)\bigr)\bigr]
\\[-1pt]
&&\qquad =\sum_ { \varnothing\neq K \subseteq I }{\mathbb{E}_\pi} \sum
_{s \le t\dvtx  Z_K(s)=0, Z_{\{1,\ldots,J\} \setminus
K}(s)>0} \bigl[O_I f\bigl({Z}(s),A(s)
\bigr)
\\[-3pt]
&&\hspace*{194.5pt}{} - O_I f\bigl({Z}(s),A(s-)\bigr)\bigr].
\end{eqnarray*}
Therefore, to show (\ref{eqnojumpterm}) it suffices to show for
each nonempty set $K \subseteq I$, we have
%
\begin{eqnarray}\label{eqfIvanish}
&& {\mathbb{E}_\pi} \sum_{s \le t\dvtx  Z_K(s)=0, Z_{\{1,\ldots,J\}
\setminus K}(s)>0}
\bigl[O_I f\bigl({Z}(s),A(s)\bigr)
\nonumber\\[-10pt]\\[-10pt]
&&\hspace*{130pt}{} - O_I f \bigl({Z}(s),A(s-)\bigr)\bigr]=0.\nonumber
\end{eqnarray}
To prove (\ref{eqfIvanish}) we first deduce from
Definition~\ref{defaugskorohod} that when $Z_K(s)=0$ and $Z_{\{
1,\ldots,J\}\setminus K}(s)>0$,\vspace*{-2pt}
\[
Q_K\bigl(A(s-)\bigr)=A(s).
\]
Next, since $K \subseteq I$, we use the projection property of the
operator $Q_I$ to obtain
\[
Q_I\bigl(A(s)\bigr)= Q_I \bigl(Q_K\bigl(A(s-)
\bigr)\bigr)=Q_I\bigl(A(s-)\bigr).
\]
Now (\ref{eqfIvanish}) readily follows from the definition of
$O_I $ as in (\ref{eqOIf}).
Thus we have completed the proof of the lemma.
\end{pf}

\subsection{Proofs of Theorem~\texorpdfstring{\protect\ref{thmbarprimenew}}{3} and Corollary~\texorpdfstring{\protect\ref{corbarprime}}{1}}\label{secproofThmCor}
We now prove Theorem~\ref{thmbarprimenew} and Corollary~\ref{corbarprime}.

\begin{pf*}{Proof of Theorem~\ref{thmbarprimenew}}
We rewrite the jump term in (\ref{eqbarprimepre}) using the jump measures.
In view of Lemmas~\ref{lemboundary} and \ref{lemjumpdir},
\begin{eqnarray*}
\hspace*{-4pt}&&  \mathbb{E}_\pi\sum_{s\le1}
\bigl[f\bigl(Z(s),A(s)\bigr)-f\bigl(Z(s),A(s-)\bigr)\bigr]
\\[-1pt]
\hspace*{-4pt}&&\!\qquad = \sum_{\varnothing\neq K\subseteq\{1,\ldots,J\}} \mathbb{E}_\pi\!\sum
_{s\le1\dvtx  Z_K(s)=0, Z_{K^c}(s)>0} \bigl[f\bigl(Z|_{K^c}(s),A(s)\bigr)
\\[-3pt]
\hspace*{-4pt}&&\!\hspace*{196pt}{} - f\bigl(Z|_{K^c}(s),A(s-)\bigr)\bigr]
\\
\hspace*{-4pt}&&\!\qquad = \sum_{\varnothing\neq K\subseteq\{1,\ldots,J\}} \mathbb{E}_\pi\!\sum
_{s\le1\dvtx  Z_K(s)=0, Z_{K^c}(s)>0, A(s)\neq A(s-)} \bigl[f\bigl(Z|_{K^c}(s),Q_K
\bigl(A(s-)\bigr)\bigr)
\\
\hspace*{-4pt}&&\!\hspace*{258pt}{} - f\bigl(Z|_{K^c}(s),A(s-)\bigr)\bigr]
\\
\hspace*{-4pt}&&\!\qquad =\sum_{\varnothing\neq K\subseteq\{1,\ldots,J\}} \int_{z_{K^c},a}
\bigl[f\bigl(z|_{K^c},Q_K(a)\bigr)-f(z|_{K^c},a)
\bigr] \,du_K(z_{K^c},a).
\end{eqnarray*}
Thus Theorem~\ref{thmbarprimenew} follows from (\ref{eqbarprimepre}).
\end{pf*}

\begin{pf*}{Proof of Corollary~\ref{corbarprime}}
Equation (\ref{eqbarprimedisag}) immediately follows from
(\ref{eqbarprimepre}) and Lemma~\ref{lemnojumpterm}. Summing all
the equations in (\ref{eqbarprimedisag}) over the sets $I \subseteq
\{1, \ldots, J\}$, we obtain (\ref{eqbarprime}).
\end{pf*}

\section{Jump measures}\label{appjumpmeasure}
In this section, we further investigate the jump term in~(\ref{eqitopre}),
resulting in a characterization of jump measures $u_I$ in terms of the
stationary distribution $\pi$.
We start with an auxiliary result on the measures~$u_I$.

\begin{lemma}
\label{lemsupportu}
For each $I\subseteq\{1,\ldots,J\}, I\neq\varnothing$ and $k=1,\ldots,J$,
we have $u_I (\{(z_{I^c},a)\dvtx  a_k=0\}) = 0$.
\end{lemma}
\begin{pf}
We exploit the dynamics of the augmented Skorohod problem.
Since $A_k(s-)=0$ implies $Z_k(s)=0$, we have $u_I(\{(z_{I^c},a)\dvtx
a_k=0\}) = 0$ for $k\in I^c$.
We next consider $k\in I$. Since the continuous part of $A_k^k$ is
strictly increasing when $Z_k>0$,
the only possibility for $Z_I(s)=0$, $A_k(s-)=0$, and $A(s)\neq A(s-)$
to occur simultaneously is
for $Z$ to hit the face $z_I=0$ without having left the face $z_k=0$
for some positive amount of time.
Since the time $Z$ spends on the boundary has Lebesgue measure zero,
this cannot happen almost surely.\vadjust{\goodbreak}
\end{pf}

To proceed with our description of the measures $u_I$, we need tools
from theory of distributions (or generalized
functions).
For background on this theory, see \citet{duistermaatdistributions2010}, \citet{Rudin-func}.
For $I\subseteq\{1,\ldots,J\}$, we define the operator $T_I^*$ on
distributions through
\[
T^*_If=\frac{1}{2}\sum_{i,j\in I}
\frac{{{\partial^2}}}{{\partial
{z_i}\partial{z_j}}} \bigl[\Sigma_{ij}(\cdot)f \bigr] 
- \sum_{j \in I} \theta_j
\frac{\partial}{{\partial
{z_j}}}f - \operatorname{tr}(\nabla_a f)
\]
for any distribution $f$.
With the understanding that we identify any probability measure with
the distribution it generates,
we can differentiate (probability) measures and $T_I^*$ can act on measures.
We also define
\[
d\pi_{I} (z_{I^c},a)= \int_{z_I} \,d
\pi(z,a).
\]
The main result of this section is that $u_I$ can be expressed in terms
of $\pi$.
Indeed, together with Lemma~\ref{lemsupportu}, it completely
determines $u_I$.

\begin{proposition}
For each $I\subseteq\{1,\ldots,J\}, I\neq\varnothing$, we have,
with $z_{I^c}\in(0,\infty)^{|I^c|},a\in\mathbb M^{J\times J}_+$ and
$a_k\neq0$ for $k=1,\ldots,J$,
\[
du_I(z_{I^c},a) = \sum_{K\subseteq I, K\neq\varnothing}
(-1)^{|I\setminus K|} \int_{z_{I\setminus K}} \bigl[T^*_{K^c} \,d\pi
_{K}\bigr](z_{K^c},a).
\]
\end{proposition}
\begin{pf}
Equation~(\ref{eqbarprimepre}) forms the basis of the proof, together
with the identity
\begin{eqnarray*}
&&  \mathbb{E}_\pi\sum_{s\le1}
\bigl[f\bigl(Z(s),A(s)\bigr)-f\bigl(Z(s),A(s-)\bigr)\bigr]
\\
&&\qquad =\sum_{\varnothing\neq K\subseteq\{1,\ldots,J\}} \int_{z_{K^c},a}
\bigl[f\bigl(z|_{K^c},Q_K(a)\bigr)-f(z|_{K^c},a)
\bigr] \,du_K(z_{K^c},a),
\end{eqnarray*}
which was established in Section~\ref{secproofThmCor}.
Fix some nonempty $I\subseteq\{1,\ldots,J\}$.
For $f\in C^2_b(\mathbb{R}_+^J\times\mathbb M_+^{J\times J})$ with
the property that $f$
vanishes on $\bigcup_{i\in I^c} F_i \cup\bigcup_i F_i^a$, (\ref
{eqbarprimepre}) reduces to
\[
\int_{\mathbb{R}_+^J\times\mathbb M^{J\times J}_+} Tf(z,a) \,d\pi(z,a)= \sum
_{L\subseteq I\dvtx L\neq I} \int_{z_{I^c\cup L},a} f(z|_{I^c \cup
L},a)
\,du_{I\setminus L}(z_{I^c\cup L},a).
\]
If moreover $f(z,a)$ does not depend on $z_I$, this can be simplified further,
%
\begin{eqnarray}\label{eqpiandu}
&& \int_{\mathbb{R}_+^J\times\mathbb M^{J\times J}_+} Tf(z,a) \,d\pi(z,a)
\nonumber\\[-8pt]\\[-8pt]
&&\qquad = \sum_{L\subseteq I\dvtx L\neq I} \int_{z_{I^c},a} f(z|_{I^c},a)
\int_{z_L} \,du_{I\setminus L}(z_{I^c\cup L},a).\nonumber
\end{eqnarray}
The left-hand side can be rewritten using the theory of differentiation
for distributions
\citet{duistermaatdistributions2010}, Chapter~4, or \citet{Rudin-func}, Section~II.6.12.
This leads to
\[
\int_{\mathbb{R}^J_+ \times\mathbb M^{J\times J}} T f(z,a) \,d\pi (z,a) = \int_{z_{I^c},a}
f(z|_{I^c},a) \bigl[T^*_{I^c} \,d\pi_I
\bigr](z_{I^c},a).
\]
Combining this with (\ref{eqpiandu}) and rearranging terms, we get
\begin{eqnarray*}
&& \int_{z_{I^c},a} f(z|_{I^c},a)
\,du_I(z_{I^c},a)
\\
&&\qquad = \int_{z_{I^c},a} f(z|_{I^c},a)\bigl[T^*_{I^c}
\,d\pi_I\bigr] (z_{I^c},a)
\\
&&\quad\qquad{} - \sum
_{L\subseteq I\dvtx L\neq\varnothing, L\neq I} \int_{z_{I^c},a} f(z|_{I^c},a)
\int_{z_L} \,du_{I\setminus L} (z_{I^c\cup L},a).
\end{eqnarray*}
This shows that, for $z_{I^c}\in(0,\infty)^{|I^c|},a\in\mathbb
M^{J\times J}_+$ and $a_k\neq0$ for $k=1,\ldots,J$,
\[
du_I(z_{I^c},a) = T^*_{I^c} \,d\pi_I
(z_{I^c},a) - \sum_{L\subseteq I,L\neq\varnothing,
L\neq I} \int
_{z_L} \,du_{I\setminus L} (z_{I^c\cup L},a).
\]
Since $|I\setminus L|<|I|$, this representation allows us to finish
the proof
of the proposition by an elementary induction argument on $|I|$.
Alternatively, one could use a version of the inclusion-exclusion
principle \citet{stanleyec1}, Section~2.1.
\end{pf}


\begin{appendix}
\section{Proof of (\texorpdfstring{\lowercase{\protect\ref{eq1dlap}}}{2.5})}\label{apppiLaplace}
This appendix uses Theorem~\ref{thmonedim} to find the Laplace
transform of the stationary
distribution $\pi$ of $(Z,A)$ in the one-dimensional case,
thereby showing in particular that Theorem~\ref{thmonedim} completely
determines $\pi$.
Writing $\mathcal L(\alpha,\eta)$ for the Laplace transform of
$\pi$, Theorem~\ref{thmonedim} implies that
%
\begin{equation}
\label{eqBARlapl} \bigl(\tfrac{1}2\sigma^2\alpha^2
-\alpha\theta-\eta \bigr)\mathcal L(\alpha,\eta) +\eta\mathcal L(0,\eta) +\alpha
\theta=0.
\end{equation}
In particular, on setting $\eta=\frac{1}2\sigma^2\alpha^2 -\alpha
\theta$ we get
\[
\bigl[\tfrac{1}2\sigma^2\alpha^2 -\alpha\theta
\bigr]\mathcal L \bigl(0,\tfrac{1}2\sigma^2
\alpha^2 -\alpha\theta \bigr) +\alpha \theta=0.
\]
After substitution of $\alpha=(\theta+\sqrt{\theta^2+
2\sigma^2\eta})/\sigma^2$, we find that
\[
\eta\mathcal L(0,\eta) = -\theta \biggl[\frac{\theta+\sqrt{\theta
^2+2\sigma^2\eta}}{\sigma^2} \biggr].
\]
Substituting this back into (\ref{eqBARlapl}) and simplifying
the resulting expression, we obtain the Laplace transform given in
(\ref{eq1dlap}).

\section{The augmented Skorohod problem and~uniqueness}\label{appuniqueness}
In this appendix, we prove that the augmented Skorohod problem
admits a unique solution.
To this end, we employ a similar contraction map as in Lemma~3.6 of
\citet{Mandel-ramanan}.
Define a map $\Lambda$ from $\mathbb{D}^{J \times J}$ to $\mathbb
{D}^{J \times J}$ by setting, for $t\ge0$,
%
\begin{equation}
\label{eqmapLambda} {\Lambda(b )}_i^j(t)= \sup
_{s \in{\Phi_{(i)}}(t)} \bigl[ \chi_i^j(s) + \bigl[
\widetilde P b^j \bigr]_i(s)\bigr].
\end{equation}
Momentarily we show that $\Lambda$ is a contraction map, and
thus $\Lambda$ has a unique fixed point $b$. This also implies that,
defining $b^{(0)}=0$ and $b^{(n)}={\Lambda(b^{(n-1)} )}$ for $n\ge1$,
we have $\|{b^{(n)}- b}\|_T \rightarrow0$ as $n
\rightarrow\infty$ for every $T>0$.
Here and throughout this proof, we write $\|x\|_T=\sup_{t\in[0,T]}
|x(t)|$; this
should not be confused with the \mbox{1-}norm and 2-norm used elsewhere in
this paper.
Since $\chi$ is nonnegative and nondecreasing
and $\widetilde P$ is nonnegative, we deduce that $b^{(n)}$ is
componentwise nonnegative and nondecreasing for each $n$.
Therefore, we obtain that the fixed point $b$ is also nonnegative and
nondecreasing. Now let $a = \chi-\widetilde R b$, $z=\Gamma(x)$, and
$y=\Phi(x)$. We now verify directly that $(z, y, a, b)$ is a
solution to the augmented Skorohod problem.
Only the fourth and fifth requirement in Definition~\ref{defaugskorohod} are not immediate.
The fourth requirement can be shown to hold using the same argument as
in the proof of Lemma~\ref{lemZgammacomp}.
For the fifth requirement, we note that if $z_i(t)=0$,
(\ref{eqmapLambda}) implies that for each $j$,
\[
b^j_i(t)=\chi^j_i(t)+\bigl(
\widetilde P b^j\bigr)_i(t),
\]
which yields
\[
a_i(t)=\chi_i(t)-(\widetilde R b)_i(t)=
\chi_i(t)+(\widetilde P b)_i(t)-b_i(t)=0.
\]

To establish the uniqueness of solutions to the augmented Skorohod
problem, we use the contraction map $\Lambda$. Suppose $(z, y,
a, b)$ solves the augmented Skorohod problem. Let $\tilde b= \Lambda
(b)$. If we can show that $\tilde b= b$, meaning $b$ is a fixed
point of $\Lambda$, then it follows from the uniqueness of the fixed
point that there must be a unique solution to the augmented Skorohod
problem. Suppose there exists some $i,j$ and $t_0$ such that
$\tilde b_i^j (t_0) \ne b_i^j (t_0)$. We discuss two cases. If
$z_i(t_0)=0$, using nonnegativity and monotonicity of $b$, one can
check from (\ref{eqmapLambda}) that $\tilde b_i^j
(t_0)=\chi_i^j(t_0) + [\widetilde P b^j ]_i(t_0)$. From the
definition of the augmented Skorohod problem, we also know that
$z_i(t_0)=0$ implies $a_i^j(t_0)=\chi_i^j(t_0) + [\widetilde R b^j
]_i(t_0)=0$. Therefore, we have $\tilde b_i^j (t_0) = b_i^j
(t_0)$, a contradiction. Now consider the second case
where we have $z_i(t_0)>0$. If the set $\Phi_{(i)}(t_0)$ is empty, we
have $\tilde b_i^j (t_0)=b_i^j (t_0)=b_i^j(0)=0$. If not, let
$s$ be the maximal element in $\Phi_{(i)}(t_0)$. We deduce from the
previous case in conjunction with the complementarily condition
(\ref{eqzdeltacomplem}) that $b_i^j(t_0)=b_i^j(s)= \tilde
b_i^j(s)= \tilde b_i^j(t_0)$. This is again a contradiction.
Therefore, we obtain $\tilde b=b$ and infer that the augmented
Skorohod problem has a unique solution.

It remains to show that $\Lambda$ is a contraction map on $\mathbb
{D}^{ J \times
J}$, which is equipped with the uniform norm on compact sets. As in
the proof of Lemma~3.6 in \citet{Mandel-ramanan} we assume that,
without loss of generality, the maximum row sum of $\widetilde P$ is
$\eta<1$. It is easy to verify that for any fixed $T>0$,
\[
\bigl\|\Lambda(b)- \Lambda\bigl(b'\bigr)\bigr\|_T \le\eta\bigl\|b -
b'\bigr\|_T
\]
for all $b, b' \in\mathbb{D}^{ J \times J}$. Thus we have proved the
existence and uniqueness of a fixed point for $\Lambda$.
\end{appendix}

\section*{Acknowledgments}
We thank Mark Squillante and Soumyadip Ghosh for valuable discussions,
and Kavita Ramanan for comments on an earlier draft.
The comments of the referees have led to significant improvements,
and we thank them for their careful reading.
Part of this research was carried out while
ABD enjoyed the hospitality of the Korteweg-de Vries Institute.


%

\printaddresses

\end{document}